\documentclass[12pt]{amsart}
\usepackage{amssymb,epsf,amsmath,amscd,mathrsfs}

\usepackage{mathtools} 

\newtheorem{thm}{Theorem}[section]
\newtheorem*{thm*}{Theorem}
\newtheorem{prop}[thm]{Proposition}
\newtheorem{lem}[thm]{Lemma}
\newtheorem{question}[thm]{Question}

\newcommand{\surf}{S} 

\newcommand{\Mod}{\mathsf{Mod}}
\newcommand{\MCG}{\Mod(\surf)} 
\newcommand\Fricke{\mathfrak{F}(\surf)}   
\newcommand\mm{\mathfrak{M}}
\newcommand{\RMS}{\mm(\surf)} 
\newcommand\TT{\mathfrak{T}}
\newcommand{\Teich}{\TT(\surf)} 

\newcommand\PGL{\o{PGL}}

\newcommand{\tG}{\widetilde{G}}

\renewcommand{\o}{\mathsf}

\newcommand\R{\mathbb R}      
\newcommand\Z{\mathbb Z}      
\newcommand\C{\mathbb C}

\newcommand\Aut{\o{Aut}}    
\newcommand\Out{\o{Out}}

\newcommand{\Rep}{\mathsf{Rep}}
\newcommand{\RepG}{\Rep(\pi,G)}     
\newcommand{\RepGt}{\RepG_\tau}    
\newcommand\tM{\widetilde M}
\newcommand\dev{\o{dev}}
\newcommand\hol{\o{hol}}
\newcommand\Inn{\o{Inn}}
\newcommand\Hom{\o{Hom}}

\newcommand{\DGXS}{\mathfrak{D}_{(G,X)}(\surf)}  

\newcommand{\SLtR}{\SL(2,\R)}
\newcommand{\SLtC}{\mathsf{SL}(2,\C)}
\newcommand{\Uone}{\mathsf{U}(1)}

\newcommand{\SUtwo}{\mathsf{SU}(2)}

\newcommand{\SL}{\mathsf{SL}}
\newcommand{\Htwo}{\mathsf{H}^2}

\newcommand{\PGLtR}{\PGL(2,\R)}
\newcommand{\Ker}{\o{Ker}}
\newcommand{\inv}{^{-1}}
\newcommand{\comp}{c}

\newcommand{\EGS}{\mathfrak{E}_G(\surf)} 
\newcommand{\FolGS}{\mathscr{F}_G(\surf)}   
\newcommand{\A}{\mathsf{A}}
\newcommand{\U}{\mathbb{U}}
\newcommand{\UTk}{\U^\comp\Teich}
\newcommand{\URk}{\U^\comp\RMS}
\newcommand{\UE}{\U^\comp_\tau\EGS}

\newcommand{\MV}{\mu^\comp}

\newcommand{\WP}{\mathsf{WP}}
\newcommand{\UTWP}{\U^\WP\Teich} 
\newcommand{\UWPS}{\U^{\WP}\RMS}
\newcommand{\UEWP}{\U^{\WP}_\tau\EGS}

\begin{document}
\title[Mixing Moduli Spaces and Flat Bundles]
{Mixing flows on moduli spaces of flat bundles over surfaces}

\author[Forni]{Giovanni Forni}
\address{Department of Mathematics, University of Maryland,
College Park, MD 20742 USA}
\email[Forni]{gforni@math.umd.edu}
\author[Goldman]{William M. Goldman}
\email[Goldman]{wmg@math.umd.edu}

\dedicatory{\centerline{Dedicated to Nigel Hitchin on his seventieth birthday}}

\setcounter{tocdepth}{3}

\date{\today}

\begin{abstract}
We extend Teichm\"uller dynamics to a flow on the total space of a flat bundle of 
deformation spaces $\RepG$ of representations of the fundamental group $\pi$
of a fixed surface $\surf$. 
The resulting dynamical system is a continuous version of the action of the mapping class group $\MCG$ of $\surf$ on $\RepG$.
We observe how ergodic properties of the $\MCG$-action relate to this flow. 
When $G$ is compact, this flow is strongly mixing over each component of 
$\RepG$ and of each stratum of the Teichm\"uller unit sphere bundle over the Riemann
moduli space $\RMS$. We prove ergodicity for the analogous lift of the Weil-Petersson
geodesic local. flow.
\end{abstract}

\thanks{
The authors gratefully acknowledge research support from NSF Grants 
DMS1406281,
and  DMS-1600687 
as well as the Research Network in the Mathematical Sciences DMS1107367 (GEAR). 
The author(s) also would like to thank the Isaac Newton Institute for Mathematical Sciences, Cambridge, 
for support and hospitality during the programme 
{\em Nonpositive curvature group actions and cohomology\/}  where this paper was completed. 
This programme was supported by EPSRC grant no EP/K032208/1.
}

\maketitle

\pagebreak
\tableofcontents
\thispagestyle{empty}

\section{Introduction}
This note develops a family of dynamical systems arising from
moduli problems in low-dimensional geometry and gauge theory. 
Moduli spaces of flat connections (or equivalently of representations of the
fundamental group) are one example.
Another example arises from the classification of {\em locally homogeneous geometric structures.\/}
Group actions arise whose complicated dynamics suggest
that the dynamical systems themselves be viewed as the solution of the classification problem, 
rather than the quotient {\em moduli spaces\/} which may be intractable.

This classification is modeled on the classification of Riemann surfaces by the {\em Riemann moduli space $\RMS$,\/}
the points of which correspond to biholomorphism classes of Riemann surfaces with fixed topology $S$.
Although $\RMS$ is complex algebraic variety, 
it has proved useful to regard $\RMS$ as the quotient of 
a more tractable complex manifold, the {\em Teichm\"uller space $\Teich$,\/}
by the mapping class group $\MCG$, 
a discrete group of biholomorphisms acting properly on $\Teich$.

This paper initiates a general program for analyzing these classification problems.
Earlier work (see, for example, Goldman~\cite{MR2264541}) developed the study of the action of 
$\MCG$ on deformation spaces $\RepG$ of representations, 
where $\pi = \pi_1(S)$.
In some cases, the action is proper, with a tractable quotient.
Such is the case of the component $\Fricke$ (when $G=\PGLtR$) of marked hyperbolic structures on $S$,
where the uniformization theorem identifies the quotient with $\RMS$.
In other cases (such as when $G$ is compact) the action is chaotic,
exhibiting dynamical complexity. 
Here we propose replacing this discrete group action by an action of a Lie group 
(either $\SLtR$ or its subgroup $\A$ consisting of diagonal matrices), 
and exploiting deep results from the well-developed theory of Teichm\"uller dynamics to obtain finer dynamical information.  

The main new result of this paper is the following observation.

Let $G$ be a 
Lie group, and $\surf$ a closed orientable hyperbolic surface with fundamental group $\pi$. 
Choose a connected component  $\RepGt$ of the space $\RepG$, 
and a connected component $\URk$ of a stratum of the unit-Teichm\"uller sphere bundle $\U\RMS$ over the Riemann moduli space $\RMS$.
Denote the natural flat $\RepGt$-bundle over $\URk$ by $\UE$.
Horizontally lift the Teichm\"uller geodesic flow from $\UTk$ to a flow $\Phi$ on $\UE$, 
and consider its restriction $\Phi^c$ to the connected component $\URk$.

Recall that a flow $\Phi$ on a probability space $(X,\mu)$ is {\em strongly mixing\/}
if for all measurable sets $A,B\subset X$, 
\[
\lim_{t\to\infty} \mu(A\cap \Phi_t(B)) = \mu(A)\mu(B). \]
This is a stronger condition than weakly mixing, which in turn is stronger than ergodicity.
For convenience of the reader, we recall that $\Phi$ is {\em weakly mixing\/} if
if for all measurable sets $A,B\subset X$, 
\[
\lim_{T\to\infty} \frac{1}{T} \int_0^T \Big\vert \mu(A\cap \Phi_t(B)) - \mu(A)\mu(B) \Big\vert dt  = 0, \]
and is {\em ergodic\/} if
\[
\lim_{T\to\infty} \Big( \frac{1}{T} \int_0^T \mu(A\cap \Phi_t(B)) - \mu(A)\mu(B) \Big) dt  = 0. \]
By the Birkoff-Khinchin pointwise ergodic theorem, ergodicity is equivalent to
the condition that every $\Phi$-invariant measurable subset either has measure zero
or its complement has measure zero.
One consequence of ergodicity is that if an invariant measure $\mu$ is ergodic, 
then the orbit of $x$ in the support of $\mu$ is dense, for $\mu$-almost every $x$.
(For further background see one of the standard texts, such as Katok-Hasselblatt~\cite{MR1326374},
Petersen~\cite{MR833286} or Walters~\cite{MR648108}.)

\begin{thm*}
Suppose $G$ is a compact connected Lie group.
Consider a connected component of a stratum $\URk$ of $\U\RMS$ and a connected component
$\RepGt$ of $\RepG$. 
The flow $\Phi^c$ on $\UE$  is strongly mixing with respect to a smooth invariant probability measure on $\UE$. 
\end{thm*}
\noindent
The invariant smooth probability measure on $\UE$ is the 
measure induced from the symplectic measure $\nu$ on the fiber and the Masur-Veech measure $\MV$ on the base.

Similar techniques enable us to deduce ergodicity for the analogous lift of the Weil-Petersson
geodesic local flow:
\begin{thm*}
The horizontal local flow obtained by lifting the Weil-Petersoon geodesic flow from the 
Weil-Petersson unit tangent bundle $\UEWP$ to the flat $\RepGt$-bundle  
$\UEWP$ is ergodic with respect to a smooth invariant probability measure.
\end{thm*}

%

Finally, we dedicate this paper to Professor Nigel J.\ Hitchin, 
whose pioneering work on moduli spaces has had a profound effect on this field.
In particular, his paper~\cite{MR1065677} pointed the way to use the variation of the refined
geometry of $\RepG$ over Teichm\"uller space to obtain finite-dimensional representations
of $\MCG$.

\section{Classification of Riemann surfaces: Riemann's moduli space}
A prototype of the classification of these structures is the analogous theory of moduli of Riemann surfaces.
Begin by fixing a fixed topological surface $\surf.$
The points of the {\em Riemann moduli space\/} $\RMS$ 
parametrize the different complex structures on $\surf$,
that is, Riemann surfaces having $\surf$ as the underlying topology.
Although $\RMS$ is not generally a manifold, 
it may be understood as the quotient of the {\em Teichm\"uller space\/} $\TT(\surf)$ by the 
{\em mapping class group\/} $\MCG$. 
Recall that the {\em Teichm\"uller space\/} $\TT(\surf)$ comprises equivalence classes of 
Riemann surfaces $M$, 
together with a homotopy class of homeomorphisms $\surf \xrightarrow{~f~} M$ 
(called {\em markings\/}). 
Marked Riemann surfaces $(M_1,f_1)$ and $(M_2,f_2)$ are {\em equivalent\/}
if there is a biholomorphism $M_1 \xrightarrow{~\phi~} M_2$ such that $\phi\circ f_1$
is isotopic to $f_2$. 
The {\em marked Riemann surfaces\/} $(M,f)$ are considerably easier to study:
for example $\TT(\surf)$ is a contractible complex manifold.
Then $\RMS$ is the quotient of $\TT(\surf)$ by the {\em mapping class group\/}
\[
\MCG := \pi_0 \big( \o{Homeo}(\surf) \big). \]

Riemann is credited for considering the {\em moduli space\/} $\RMS$ 
of Riemann surfaces of fixed topology, 
which is now known to be a quasi-projective variety over $\C$.  
Riemann computed its dimension, and understood its local structure,
long before the modern context and terminology had been developed.
Due to Riemann surfaces with nontrivial automorphisms, 
$\RMS$ fails to be a manifold.
Nonetheless it enjoys the structure of a complex orbifold 
(or {\em $V$-manifold\/} in the sense of Satake~\cite{MR0079769}); 
in particular its underlying topology is Hausdorff.

The classification of Riemann surfaces with the extra structure of a marking
forces us to leave the realm of complex algebraic varieties,
but it provides a more tractable answer, 
since $\Teich$ is a complex manifold.
Riemann's moduli space  $\RMS$ is then the quotient of $\Teich$ by $\MCG$.
Expressing $\RMS$ as a quotient requires the {\em properness\/} of the $\MCG$-action on $\Teich$, 
a result attributed to Fricke at the turn of the twentieth century
(although the context for these types of questions seems not to have been developed at that time).

For background on Teichm\"uller theory, we recommend Hubbard~\cite{MR2245223}.

\section{Classification via deformation spaces of flat structures}
This section summarizes the motivation for our study,
namely the classification of locally homogeneous geometric structures,
and the closely related classification of flat connections.
In many interesting and important cases, 
the classification reduces to a tractable {\em moduli space,\/}
analogous to Riemann's moduli space of Riemann surfaces of fixed topology.
However, in general, this construction leads to a non-Hausdorff quotient.
We therefore shift our attention to the dynamical system arising from the mapping class
group action on the space of {\em marked\/} structures.

\subsection{Ehresmann-Thurston geometric structures}
The study of {\em locally homogeneous geometric manifolds\/} was initiated
by Ehresmann in his 1936 paper~\cite{ehresmann1936espaces}. 
(See \cite{MR2827816} for a modern general survey,
and \cite{Hist} for a historical account.)
These structures on a manifold $M$ are
defined by local coordinates taking values in a model manifold $X$ which enjoys 
a transitive action of a Lie group $G$. 
The $G$-invariant geometry on its homogeneous space $X$ then transplants
locally to $M$. 
A familiar example a {\em flat Riemannian metric,\/}
or a {\em Euclidean structure,\/} where $X$ is Euclidean space
and $G$ is its group of isometries. 
When $G$ is extended to the group of affine automorphisms,
one obtains {\em (flat) affine structures,\/} which are more traditionally described
as flat torsionfree affine connections. 

A natural question is, given a geometry $(G,X)$ and a topology $\surf$,
{\em classify all the geometric structures on $\surf$ modeled on the $G$-invariant
geometry of $X$.\/} 
One would hope for a {\em moduli space\/} for $(G,X)$-structures on $\surf$
analogous to Riemann's moduli space $\RMS$. 

Such {\em Ehresmann structures\/} closely relate to the fundamental group $\pi = \pi_1(\surf)$. 
The charts in the coordinate atlas globalize to a local homeomorphism 
$\tM \xrightarrow{~\dev~} X$,
called the {\em developing map\/} and the coordinate changes globalize to a compatible
{\em holonomy homomorphism\/} $\pi\xrightarrow{~h~} G$. 
The pair $(\dev,h)$ determines the structure, and 
$h$ is unique up to the action of the group $\Inn(G)$ of inner automorphisms of $G$.

The {\em deformation space\/}  $\DGXS$ consists of equivalence
consists of equivalence classes of {\em marked $(G,X)$-structures\/} on $\surf$,
analogous to Teichm\"uller space. 
Associating to a marked $(G,X)$-manifold 
$\surf\xrightarrow{~f~} M$ the conjugacy class of its
{\em holonomy homomorphism\/} $\pi_1(\surf)\xrightarrow{~h~} G$
defines a mapping 
\[
\DGXS 
\xrightarrow{~\mathsf{hol}~}  \mathsf{Rep}(\pi,G).\]
The {\em Ehresmann-Weil-Thurston\/} principle asserts that $\hol$  is (essentially) a local homeomorphism,
with respect to a suitable natural topology  on $\DGXS$,
namely the quotient topology induced from the $C^r$-topology on developing maps,
for $1\le r \le \infty$.
(See \cite{MR2827816,Hist} and the references there for more details.)

\subsection{Flat connections}
The classification of flat connections or {\em flat bundles\/} is similar.
Flat connections on a fixed bundle correspond to a reduction of the structure group
to the discrete topology, in which case the classification corresponds to that of 
representations $\pi_1(\surf) \longrightarrow G$, {\em up to conjugacy.\/}
Taking $\pi$ to be $\pi_1(\surf)$ (or any finitely generated group), 
and $G$ a Lie group, the space $\Hom(\pi, G)$ admits the structure of an $\R$-analytic set.
If $G$ is an algebraic group of matrices, then $\Hom(\pi, G)$ is an affine algebraic set.
Furthermore $\Aut(\pi)\times\Aut(G)$ acts on  $\Hom(\pi, G)$ preserving this structure.

The space of flat $G$-bundles is the quotient space of $\Hom(\pi, G)$ 
by the subgroup $\{1\}\times \Inn(G)$, 
denoted 
\[ \RepG := \Hom(\pi,G)/\Inn(G). \]
Here $\RepG$ is given the quotient topology induced from the classical topology on
$\Hom(\pi,G)$. 
In general $\Inn(G)$ fails to act properly, and $\RepG$ is not Hausdorff. 
Its {\em maximal Hausdorff quotient,\/} in many cases is the Geometric-Invariant-Theory quotient 
$\Hom(\pi,G)//G$, and called the {\em character variety.\/}
We refer to Sikora~\cite{MR2931326} and Labourie~\cite{MR3155540}
for further details.

$\Inn(\pi)$ acts trivially as its action is absorbed by the action of $\Inn(G)$.
Therefore $\RepG$ admits a natural action of the {\em outer automorphism group\/}
\[
\Out(\pi) := \Aut(\pi)/\Inn(\pi). \]
When $\pi$ is the fundamental group $\pi_1(\surf)$ of a surface $\surf$,
then the mapping class group $\MCG$ embeds in $\Out(\pi)$ and therefore acts on $\RepG$. 

\subsection{Hyperbolic geometry on surfaces}\label{sec:HypGeo}
Consider the special case when $G\cong\PGLtR$ is the group of isometries of the hyperbolic plane
$X=\Htwo$.
In this special case, Weil~\cite{MR0137792} proved that $\hol$ embeds the 
{\em Fricke space\/} $\Fricke$ of marked hyperbolic structures on $\surf$ into
an open subset (indeed, a connected component) 
of 
$\RepG$.
The uniformization theorem 
identifies the Fricke space $\Fricke$ with the Teichm\"uller space $\TT(\surf)$.
Furthermore this identification 
$\Fricke\longleftrightarrow \TT(\surf)$ is equivariant with respect to the action of $\MCG$.
Since $\MCG$ acts properly on $\TT(\surf)$, 
its action on $\Fricke$ is proper.


Even when 
$\DGXS$ %
is a nice Hausdorff manifold, 
the $\MCG$-action can be extremely chaotic, with a highly intractable quotient.
By Baues~\cite{MR1680196,MR2181958}, 
the deformation space of complete affine structures on $\surf = T^2$
is homeomorphic to $\R^2$ and $\MCG \cong \o{GL}(2,\Z)$ 
acts on $\R^2$ by the usual linear action. 
The non-Hausdorff quotient $\R^2/\o{GL}(2,\Z)$ admits no nonconstant continuous functions. 

Analogous to the Riemann moduli space $\RMS$ is the quotient \linebreak $\DGXS/\MCG$,
which looks like $\RepG/\MCG$.
However both quotients may well be unmanageable, 
as the previous example shows. 
Rather than passing to the quotient, we propose that the classification of $(G,X)$-structures
on $\surf$ should be the {\em dynamical systems\/} given by the actions of $\MCG$
on $\DGXS$ and $\RepG$.

\section{Surface group representations}\label{sec:SurfGpReps}
The most detailed information is known when 
$\surf$ is a surface and $\pi$ is its fundamental group.
Under very general conditions~\cite{MR762512}, 
the deformation spaces admit a symplectic/Poisson geometry defined by the topology $\surf$.
This symplectic structure (which is part of a K\"ahler structure) 
was first written down by Narasimhan (unpublished), 
described when $G$ is compact by Atiyah-Bott~\cite{MR702806}
and developed for Lie groups $G$ for which the adjoint representation of $G$ is orthogonal in~\cite{MR762512}.
This geometry is invariant under $\MCG$. 
Denote the smooth measure induced by the symplectic structure by $\nu$.
When $G$ is a compact Lie group, $\nu$ is an invariant finite measure
(Huebschmann-Jeffrey~\cite{MR1277051}) and we normalize $\nu$ to be a probability
measure on each connected component of $\RepG$.

Giving $\surf$ a complex structure, $\RepG$ admits even richer structure.
For example, when $G$ is a compact Lie group, then $\RepG$ inherits a K\"ahler structure
subordinate to the symplectic structure, generalizing the structure of an abelian variety
(the Jacobian) in the simplest case when $G=\Uone$.
(\cite{MR2400111} expounds Higgs bundle theory in the ``trivial'' case of rank $1$.
See also \cite{MR2681690}.)
Fundamental is how these structures deform as a marked Riemann surface 
$\surf\to M$ varies over Teichm\"uller space $\TT(\surf)$.
This variation is the {\em continuous\/} version of the action of $\MCG$ on $\RepG$,
which we view as a {\em discrete\/} object.

For $G$ complex reductive, $\RepG$ is {\em complex symplectic,\/}
which refines to a {\em hyper-K\"ahler structure\/} for a marked Riemann surface $(f,M)$
(Hitchin~\cite{MR887284}).
These represent {\em reductions of the structure group to its maximal compact subgroup,\/}
namely $\mathsf{Sp}(2N,\R) \supset \mathsf{U}(N)$ in the real case and
$\mathsf{Sp}(2N,\C) \supset \mathsf{Sp}(2N)$ in the complex case.
In this way \[ \RepG \times \Teich \longrightarrow \Teich\]
is a family of hyper-K\"ahler manifolds over $\Teich$,
with fixed underlying complex-symplectic structure on the fibers.
The quotient 
\[ \EGS := \bigg(\RepG \times \Teich\bigg)\bigg/\MCG \]
is a flat $\RepG$-bundle over the Riemann moduli space $\RMS$:
\[ \begin{CD}
\RepG @>>> \EGS \\
@. @VVV \\
@. \RMS 
\end{CD}\]

\subsection{The isomonodromic foliation}
The foliation $\FolGS$ of $\EGS$ defining the flat structure is 
induced by the foliation of the covering space 
\[ \RepG \times \Teich \longrightarrow \EGS \]
defined by the projection
\[ \RepG \times \Teich \longrightarrow \RepG. \]
It is {\em dynamically equivalent\/} to the
$\MCG$-action  on $\RepG$ in the following sense:

\begin{prop}
Leaves of $\FolGS$ bijectively correspond to $\MCG$-orbits on $\RepG$.
\end{prop}
\noindent $\FolGS$  is a continuous object corresponding to the $\MCG$-action.
For example, $\MCG$-invariant measures on $\RepG$ correspond to 
invariant transverse measures for $\FolGS$.
However, deep properties of the geometry of $\Teich$ let us 
reduce to a dynamically equivalent action of a {\em connected Lie group.}

\subsection{Extending the Teichm\"uller flow}
Teichm\"uller defined a Finsler metric on $\Teich$ which is the natural $L^1$-metric
on holomorphic quadratic differentials, regarded as (co-)tangent vectors to $\Teich$.
The unit-sphere bundle $\U\Teich$ of $\Teich$ then admits a corresponding geodesic flow 
(which is part of an $\SLtR$-action $\phi$). 
Lift the foliation $\FolGS$ to a folation $\U\FolGS$ on the flat $\RepG$-bundle 
$\U\EGS$ over $\U\Teich$. 
This foliation $\U\FolGS$ contains  an $\SLtR$-action on $\U\EGS$,
whose restriction to the subgroup $\A\subset\SLtR$ of diagonal matrices
is the horizontal extension  (with respect to the flat connection) 
of the {\em Teichm\"uller geodesic flow\/} on $\U\Teich$.
In this way we replace the dynamics of the $\MCG$-action on $\RepG$ by actions of connected Lie groups
on $\U\EGS$.

\subsubsection{Review of Teichm\"uller theory}\label{sec:ReviewTeich}

We briefly review the Teichm\"uller flow,  
referring to Masur~\cite{MR2655331} and Forni-Matheus~\cite{MR3345837} for details and references.

The tangent space to the Teichm\"uller space $\Teich$ at a marked Riemann surface
$M$ identifies with the vector space of holomorphic quadratic differentials on $M$.
The infinitesimal Teichm\"uller metric is the Finsler metric arising from the natural $L^1$-norm 
on quadratic differentials, and we denote the unit sphere bundle by $\U\Teich$.
There is a natural stratification of $\U\Teich$ by complex submanifolds.

The strata of $\U\Teich$ are labeled by the vector of the orders of the zeroes of the holomorphic quadratic differentials 
or by the the vector of the orders of the zeroes of the holomorphic Abelian differentials, 
whenever the quadratic differentials in the stratum are squares of Abelian differentials.  
Generally the strata are disconnected.
We label the connected components  $\UTk$ by an index $c\in\pi_0\big(\U\Teich\big)$.
Kontsevich-Zorich~\cite{MR2000471} describes the connected components for strata of Abelian differentials,
and thus the connected components of strata of quadratic differentials which are squares.
Lanneau~\cite{MR2423309} describes the components of strata of quadratic differentials which are not squares.

Each component $\UTk$ is $\MCG$-invariant,  and their  quotients 
\[ 
\URk := \UTk/\MCG\] 
partition the Teichm\"uller-unit sphere bundle $\U\RMS$ of the Riemann moduli space 
$\RMS = \Teich/\MCG$. 
Masur~\cite{MR644018} and Veech~\cite{MR866707} showed that  every connected component
$\UTk$  (in fact, every stratum) carries a smooth measure $\hat \MV$, whose projection onto 
the corresponding component $\URk$ of the moduli space
has {\em finite} total mass.
We call this measure  the {\it Masur-Veech measure\/} 
and denote it $\MV$. 

A nonzero holomorphic quadratic differential $q$ on a Riemann surface $M$ determines
a conformal Euclidean structure, singular at the zeroes of $q$.
Composing the developing map $\tilde{M} \longrightarrow \R^2$
with a unimodular linear transformation $\R^2 \longrightarrow\R^2$
gives a new singular Euclidean structure, which arises from a holomorphic quadratic
differential on a Riemann surface. 
In particular this $\SLtR$-action preserves each component, 
as well as 
$\MV$.

The restriction of the $\SLtR$ to the subgroup $\A$ comprising 
positive diagonal matrices is the Teichm\"uller geodesic flow.
By \cite{MR644018, MR866707}, 
the Teichm\"uller geodesic flow on the probability space
$(\URk,\MV)$ is ergodic.
Moore~\cite{MR880376} proved an ergodic $\R$-action
on a probability space which extends to a measure-preserving action of $\SLtR$
is strongly mixing. 
Thus the Teichm\"uller geodesic flow on $(\URk,\MV)$ is strongly mixing.

Veech~\cite{MR866707}  in fact proved the stronger result that the 
Teichm\"uller flow is non-uniformly hyperbolic, 
in the sense that all Lyapunov exponents of its tangent cocycle, 
with the only exception of the one in the flow direction, 
are non-zero. 
Classical results of Pesin's theory 
(see, for example,  Corollary~11.22 of Barreira-Pesin~\cite{MR2348606}) 
then imply that the Teichm\"uller flow has the Bernoulli property with respect to the Masur-Veech measures, 
that is, it is measurably isomorphic to a Bernoulli shift, 
and in particular it has completely positive entropy.

The cleanest statement involving the measurable dynamics concerns the case when 
$\MCG$ acts ergodically on the components of $\RepG$. 
By Goldman~\cite{MR762512} and 
Pickrell-Xia~\cite{MR1915045,MR2015257},
this occurs whenever $G$ is a compact connected Lie group. 
For noncompact $G$, 
these actions display both chaotic dynamics and trivial dynamics,
and the situation is much less understood.

\subsection{Compact groups}\label{sec:CompactGroups}

As a first application of these dynamical ideas, we consider a compact Lie group $G$.

\subsubsection{Connected components of representation varieties}
The connected components of $\RepG$ 
correspond to the elements of the fundamental group 
$\pi_1\big([G,G]\big)$ when $G$ is compact or complex reductive
~\cite{MR952283, MR3155540, MR1380633,MR1206154}.
These may be understood in terms of the {\em second obstruction class\/}
of a representation $\pi\xrightarrow{~\rho~} G$, 
which is the obstruction for lifting a representation
from $G$ to its universal covering group $\tG\longrightarrow G$:
Give  $\pi$ the standard presentation
\[
\pi := \langle A_1,B_1,\dots,A_g,B_g \mid R(A_1,B_1,\dots,A_g,B_g) = 1 \rangle \]
where $R$ denotes the {\em relation\/}
\[
R(A_1,B_1,\dots,A_g,B_g) := [A_1,B_1] \dots [A_g,B_g], \]
$[A,B] := A B A\inv B\inv$, and $g$ denotes the genus of $\surf$.
Lift the images of the generators 
$\rho(A_i)\in G$ (respectively $\rho(B_i)\in G$) to 
$\widetilde{\rho(A_i)}\in \tG$ (respectively $\widetilde{\rho(B_i)\in\tG}$).
The element
\[
\mathfrak{o}(\rho) := R\bigg(
\widetilde{\rho(A_1)},\ \dots,\ 
\widetilde{\rho(B_g)} \bigg)\in \tG \]
lies in 
\[
\pi_1\big([G,G]\big) \subset \Ker \big(\tG \longrightarrow G\big)\] 
and is independent of the choice of lifts
(since lifts differ by elements of $\pi_1(G)\subset \o{Center}(G)$).
This {\em obstruction class\/} 
\[ \RepG \xrightarrow{~\mathfrak{o}} \pi_1\big([G,G]\big)\] distinguishes the connected components
of $\RepG$. 
If $\tau\in\pi_1\big([G,G]\big)$, denote the corresponding component
$\mathfrak{o}\inv(\tau)$ by $\RepG_{\tau}$.

Since 
\[
\pi_1\big([G\times G,G\times G]\big) \cong \pi_1\big([G,G]\big) \times \pi_1\big([G,G]\big),\]
for each $\tau\in\pi_1\big([G,G]\big)$, 
there is a connected component  of $\Rep(\pi,G\times G)$ corresponding to $\tau\times\tau$, 
denoted $\Rep(\pi,G\times G)_{\tau\times\tau}$.

\subsubsection{Mixing for the extended Teichm\"uller flow}
Let $\URk$ be a connected component of a stratum of 
$\U\RMS$ and $\RepGt$ be a connected component of $\RepG$.
As in \S\ref{sec:SurfGpReps}, 
form the flat $\RepGt$-bundle 
\[ \begin{CD}
\RepGt @>>> \UE \\
@. @VVV \\
@. \URk 
\end{CD}\]
over $\URk$ whose total space is the quotient 
\[
\UE := \bigg(\UTk \times\RepGt\bigg)\bigg/\MCG.\]

\begin{thm}\label{thm:Erg}
The horizontal lift of the Teichm\"uller flow to  $\UE$ is strongly mixing.
\end{thm}

\begin{proof}
The proof is essentially a combination of known results.

The first ingredient is the ergodicity of $\MCG$ on the components $\RepGt$
as in \cite{MR762512,MR1915045,MR2015257,MR2807844}.
Indeed, as noted in \cite{MR762512},
the formal property 
\[
\Hom(\pi, G\times G) = \Hom(\pi, G) \times \Hom(\pi, G), \]
its descent to quotients, and the identification of the symplectic measure
on $\Rep(\pi,G\times G)_{\tau\times\tau}$ as the product measure $\nu\times\nu$ implies {\em weak mixing\/} 
(or {\em double ergodicity\/} as in \cite{MR762512}) of $\MCG$ on $\RepGt$. 

The next ingredient is the $\SLtR$-action on $\UTk$.
This action commutes with the $\MCG$-action on $\UTk$,
and induces an action on the quotient 
\[
\URk = \UTk/\MCG. \]
and the Masur-Veech measure $\MV$ on
$\URk$ is strongly mixing with respect to the restriction of the
$\SLtR$-action to $\A$ (the Teichm\"uller geodesic flow).

The induced $\MCG \times \SLtR$-action on the product $\UTk\times\RepGt$ is
the dynamical system combining these two actions, where the $\SLtR$-factor
acts trivially on $\RepGt$. 
The product of  the lift to $\UTk$ of the Masur-Veech measure 
$\MV$ with the symplectic measure $\nu$ on $\RepGt$ 
defines an invariant smooth measure $\MV_\tau$ on 
$\UTk\times\RepGt$. 
This measure passes to a probability measure on its quotient $\UE$,
which is invariant under the induced $\SLtR$-action. 

\begin{lem}\label{lemma}
$\A$ acts ergodically on $\UE$. 
\end{lem}
\begin{proof}[Proof of Lemma~\ref{lemma}]
The proof crucially uses the {\em multiplier criterion\/} for weak mixing, 
as in in Glasner-Weiss~\cite{MR3568977}:
{\em the diagonal action on a Cartesian product of
any ergodic probability space with a weakly mixing probability space is ergodic.\/}

From the ergodicity of the action of $\A$ on $\big(\URk,\MV\big)$, proved by Masur~\cite{MR644018} and Veech~\cite{MR866707},
it follows immediately, by the definition of ergodicity, that the action of $\A\times\MCG$ on $\big(\UTk,\hat \MV\big)$ is ergodic.
Thus, by the multiplier criterion,  the weak mixing property of the $\A\times\MCG$-action on $\RepGt$ with its symplectic 
measure (and $\A$ acting trivially), discussed above, implies ergodicity of 
the action of $\A\times\MCG$ on the product  $\UTk\times\RepGt$ with the product measure. 
The quotient by the diagonal $\MCG$-action on the product $\UTk\times\RepGt$
yields an ergodic $\A$-action on $\UE$ as claimed. 

\end{proof}

\noindent
{\em Conclusion of the proof of Theorem~\ref{thm:Erg}:} 
Observe that the $\A$-action on $\UE$ is the restriction of the induced $\SLtR$-action on $\UE$.
Now, as in \S\ref{sec:ReviewTeich}, 
apply Moore's theorem~\cite{MR880376} that an ergodic $\A$-action of a probability space, which extends
to a measure-preserving action of $\SLtR$, is strongly mixing.
\end{proof}

Veech's work~\cite{MR866707} on the Teichm\"uller flow and 
Forni's work~\cite{MR1888794} on the lift of the Teichm\"uller flow to the Hodge bundle suggest:

\begin{question}\label{question}
Is the Teichm\"uller flow on $\UE$ (non-uniformly) hyperbolic with respect to the appropriate lift of the Masur-Veech measure?
\end{question}
\noindent
In this case, the flow on $\UE$ is Bernoulli, and the dynamics are completely understood.
Forni~\cite{MR1888794} answers affirmatively Question~\ref{question} in the simplest case of $G=\Uone$.
Similar questions (ergodicity, non-uniform hyperbolicity, whether the flow is Bernoulli) 
can be asked with respect to other $\SLtR$-invariant probability measures on $\UTk$.

\subsection{Noncompact groups}
The situation is much more interesting (and less well understood) for noncompact $G$.

The simplest cases are rather trivial.
As in \S\ref{sec:HypGeo}, let $G=\PGLtR$
and consider the component $\RepGt = \Fricke$ 
comprising discrete embeddings $\pi\to G$
(or, equivalently, marked hyperbolic structures on $\surf$).
Then $\MCG$ acts properly on $\RepGt$, 
with quotient corresponding to $\RMS$ by the uniformization theorem. 
Furthermore the symplectic measure $\nu$ on $\Fricke$ identifies with the Weil-Petersson
volume form on Teichm\"uller space $\Teich$ (see \cite{MR762512}).

In this case each $\MCG$-orbit in $\RepGt/\MCG$ defines a leaf of the foliation $\U\FolGS$,
which maps bijectively to $\U\RMS$ if the orbit is free (that is, if the isotropy group is trivial). 
The leaf space identifies with $\Fricke/\MCG \approx \RMS$.
The lifted $\SLtR$-action preserves these sections, 
and there is no new dynamics here.

For groups $G\supset\PGLtR$, representations in $\RepG$ close to these Fuchsian representations will also lie in open subsets upon which $\MCG$ acts properly. 
The above remarks apply in these more general cases as well.
In particular the {\em Anosov representations\/} defined by Labourie~\cite{MR2221137}
(see also Guichard-Wienhard~\cite{MR2981818}) form open subsets
upon which $\MCG$ acts properly. 
(See Burger-Iozzi-Wienhard~\cite{MR3289711} for a survey of some of this theory.)
In particular $\MCG$ acts properly on the components of $\RepG$ discovered by Hitchin~\cite{MR1174252} in the context of Higgs bundles.
For a good survey of some of these developments, 
and the closely related subject of actions of free group automorphisms on character varieties, 
see Canary~\cite{MR3382028}.

However, in many cases (such as when $G$ is a complex Lie group) 
the boundary of this open set admits a chaotic $\MCG$-action 
(Souto Storm~\cite{MR2240903}, 
Goldman~\cite{MR2026539}, 
Goldman-McShane-Stantchev-Tan~\cite{gmst},
and Maloni-Palesi~\cite{MR3420542}. 
In an important special case, 
Cantat~\cite{MR2553877} proves the existence of orbits 
(when $G=\SLtC$ and $S$ is a punctured torus) whose closure contain both $\SUtwo$-characters
and characters of discrete embeddings.
This uses work of Cantat-Loray~\cite{MR2649343}
which also relates character varieties to the sixth Painl\'eve equation.
See the excellent survey by Inaba-Iwasaki-Saito~\cite{MR2353464} and the references therein for
the closely related theory of dynamics of Painl\'eve VI.

For the other components of $\RepG$, when $G=\PGLtR$, 
it seems plausible that the $\MCG$-action is ergodic.
March\'e-Wolff~\cite{MR3457677} have proved ergodicity when $g=2$.
(Actually they show that the connected component containing the trivial representation has two ergodic components.) 
These results can then be interpreted in terms of the extended Teichm\"uller flow.

\subsection{Weil-Petersson geometry}
In a different direction, 
one can replace Teichm\"uller geometry with Weil-Petersson geometry,
obtaining a different flow with interesting dynamics.
For background on Weil-Petersson theory, 
we recommend Hubbard~\cite{MR2245223} and Wolpert~\cite{MR2641916}.
The {\em Weil-Petersson space\/} is the the complex manifold underlying $\Teich$,
but with its $\MCG$-invariant Weil-Petersson K\"ahler structure.

The unit tangent bundle $\UTWP$ of Weil-Petersson space supports the geodesic local flow
$\tilde\phi^\WP$ of the Riemannian structure underlying the Weil-Petersson K\"ahler structure. 
Since this Riemannian structure is incomplete, 
$\tilde\phi^\WP$ is only a {\em local flow.\/}
However, geodesics fail to be complete only if they converge to {\em noded Riemann surfaces.\/}
In particular the Weil-Petersson geodesic flow is defined almost everywhere for all times.

The Weil-Petersson unit tangent bundle $\UWPS$ of $\RMS$ is
the quotient $\UTWP/\MCG$ which inherits 
a local flow $\phi^\WP$ from $\tilde\phi^\WP$. 
Liouville (Riemannian) measure on $\UWPS$ defines a $\phi^\WP$-invariant probability measure on $\UWPS$.
Burns, Masur and Wilkinson~\cite{MR2993753} proved  that the Weil-Petersson geodesic flow
is Bernoulli, in particular mixing. 

Recently  Burns, Masur,  Matheus and Wilkinson proved that
the rate of mixing is at most polynomial for most strata.
However it it is rapid (super-polynomial) for exceptional strata. 
See the recent survey by Matheus~\cite{Matheus} on the dynamics of the Weil-Petersson  flow and
references therein. 

In contrast, a celebrated result by Avila, Gou\"ezel and Yoccoz \cite{MR2264836} 
(generalized by Avila and Gou\"ezel \cite{MR3071503} to all $\SLtR$-invariant probability measures)
implies the mixing rate of the Teichm\"uller geodesic flow is exponential. 

Let $G$ be a compact Lie group.
Consider the flat $\RepGt$ bundle $\UEWP$ over the unit-sphere bundle $\UWPS$ 
and, as above, horizontally lift $\phi^\WP$ to a local flow $\Phi^\WP$ on $\UEWP$. 
Combining recent results on ergodicity of the Weil-Petersson flow
(Burns-Masur-Wilkinson~\cite{MR2993753}) 
with those of $\MCG$-action on $\RepG$ implies:
\begin{thm}
The horizontal lif $\Phi^\WP$ of the Weil-Petersson geodesic local flow to
$\UEWP$ is ergodic with respect to the Lebesgue measure class.
\end{thm}
\noindent
However, due to the lack of a corresponding $\SLtR$-action 
for the Weil-Petersson flow, our methods do not prove mixing.

\bibliographystyle{amsplain}
\bibliography{MixingFlow}

\end{document}